\documentclass[12pt]{article}
\usepackage[english]{babel}
\usepackage[utf8]{inputenc}
\usepackage{amsmath}
\usepackage{graphicx}
\usepackage{amssymb}
\usepackage{amsthm}
\usepackage{tikz-cd}
\usepackage{mathrsfs}
\usepackage[colorinlistoftodos]{todonotes}
\usepackage{enumitem}
\usepackage{yfonts}
\usepackage{ dsfont }
\usepackage[utf8]{inputenc}
\title{Universality for Cokernels of Dedekind Domain Valued Random Matrices}

\author{Eric Yan}

\date{\vspace{-5ex}}

\newtheorem{thm}{Theorem}[section]

\newtheorem{defn}[thm]{Definition}

\newtheorem{lemma}[thm]{Lemma}

\newcommand{\R}{\mathbb{R}}

\newcommand{\p}{\mathfrak{p}}
\newcommand{\E}{\mathbb{E}}
\newcommand{\F}{\mathbb{F}}
\newcommand{\Z}{\mathbb{Z}}

\newcommand{\Q}{\mathbb{Q}}
\newcommand{\Hom}{\mathrm{Hom}}
\newcommand{\Mod}[1]{\ (\mathrm{mod}\ #1)}

\begin{document}

\maketitle
\thispagestyle{empty}
\footnotetext[1]{The author is grateful to Melanie Matchett Wood for direct mentorship and suggestion of the problem. The author thanks Yifeng Huang, Roger Van Peski, and the anonymous referee for very helpful comments on an earlier draft of the paper. The author was supported by the Harvard College Research Program during this research.}

\begin{abstract}
We use the moment method of Wood to study the distribution of random finite modules over a countable Dedekind domain with finite quotients, generated by taking cokernels of random $n\times n$ matrices with entries valued in the domain. Previously, Wood found that when the entries of a random $n\times n$ integral matrix are not too concentrated modulo a prime, the asymptotic distribution (as $n\to\infty$) of the cokernel matches the Cohen and Lenstra conjecture on the distribution of class groups of imaginary quadratic fields. We develop and prove a condition that produces a similar universality result for random matrices with entries valued in a countable Dedekind domain with finite quotients. 
\end{abstract}

\section{Introduction}
Wood shows in \cite{Woo19} that given any $\epsilon>0$, a random matrix $M(n)\in M_{n\times n}(\Z_p)$, with entries that lie any residue class of $\Z/p\Z$ with probability at most $1-\epsilon$, has a cokernel asymptotically distributed according to the distribution of finite abelian groups that Cohen and Lenstra conjecture \cite{CL84} to be the distribution of class groups of imaginary quadratic fields. More precisely, it was proven that for any finite abelian $p$-group $B$,
\begin{equation} \label{Zdist}
\lim_{n\to\infty}\mathbb{P}(\mathrm{cok}(M(n))\simeq B)=\frac{\prod_{k=1}^\infty(1-p^{-k})}{|\mathrm{Aut}(B)|} 
\end{equation}
when the entries $M(n)_{i,j}$ are independent and satisfy $\mathbb{P}(M(n)_{i,j}\equiv r\Mod p)\leq 1-\epsilon$ for all $r\in\Z/p\Z.$ Such random matrices $M(n)\in M_{n\times n}(\Z_p)$ satisfying this condition are called $\epsilon$-\emph{balanced}.
Random finite abelian groups can be generated by taking the quotient of a number of generators by some random relations. As such, this result regarding cokernels of random matrices is a universality result, demonstrating that the asymptotic distribution for random finite abelian groups obtained by generators and random relations is quite insensitive to the distribution of the random relations chosen (so long as they are not too concentrated in some residue class of some $\Z/p\Z$), and that this distribution matches that of Cohen-Lenstra.

This paper extends the universality of the Cohen-Lenstra distribution to more general random finite modules given by generators and random relations. The moments and distribution of random modules have rich connections with the Cohen-Lenstra-Martinet conjectures \cite{CL84,CL90} on the distribution of class groups of number fields, as explored in \cite{WW21}.

Similar to the case of finite abelian groups, a random finite module can be generated by taking the cokernel of a random $n\times n$ matrix with entries valued in some ring $T.$ A key question that needs to be answered when extending the universality result is what conditions need to be imposed on the matrix entries. We will answer this question in the case of Dedekind domains $T$ for which the quotient $T/J$ by any nonzero proper ideal $J\subset T$ is finite.

In the case of random modules over Dedekind domains, the direct application of the $\epsilon$-balanced assumption given in the random integral matrix case fails to produce a distribution similar to the one given in (1). In particular, one could ask why we cannot apply directly the existing definition of $\epsilon$-balanced, which is that for every maximal ideal $\p\subset T,$ we have $P(M(n)_{ij}\equiv r \Mod{\mathfrak{p}})\leq 1-\epsilon$ for all $r\in T/\p.$ If we take $T$ to be the ring of integers $\mathcal{O}_K$ for some number field $K,$ such a definition would imply that a $\{0,1\}$ matrix has a cokernel distribution similar to (1). Closer inspection reveals that this cannot be true. Suppose our matrix $M(n)$ is valued in $M_{n\times n}(\Z[i]).$ Then if $M(n)$ is a random $\{0,1\}$ matrix (or any matrix supported by only integers), it is invariant under the Galois action $\tau:i\mapsto -i.$ Thus $\mathrm{cok}(M(n))$ must be invariant under $\tau,$ severely restricting the set of $\Z[i]$-modules that it can possibly be isomorphic to. In other words, for any $\Z[i]$-module $N$ that is not invariant under $\tau,$ we have that \[\mathbb{P}(\mathrm{cok}(M(n))\simeq N)=0.\]

More generally, if $T=\mathcal{O}_K$ is the ring of integers of a number field $K,$ a matrix $M(n)\in M_{n\times n}(T)$ supported only on values that are invariant under some nontrivial element $\tau\in\mathrm{Gal}(K/\Q)$ cannot possibly have a cokernel distribution that is similar to (1). A sufficient condition on the entries of $M_{n\times n}(T)$ that determines that distribution of its cokernel is described in the following theorem.

\begin{thm}
Let $T$ be the ring of integers for some number field $K$ of degree $d,$ equipped with the discrete $\sigma$-algebra. Let $T_{\p}=\varprojlim T/\p^n.$ Let $\epsilon>0.$ Let $M(n)$ be a random matrix valued in $M_{n\times n}(T)$ with independent entries $M(n)_{i,j},$ such that, for every rational prime $p,$ we have $\mathbb{P}(M(n)_{i,j}\in H\Mod {p})\leq 1-\epsilon$ for every proper affine subspace $H\subset T/pT.$ Let $P=\{\p_1,\dots,\p_k\}$ be a finite set of nonzero prime ideals of $T,$ and let $T_{\p_i}=\varprojlim T/\p_i^n$ for all $i.$ Let $N$ be a finite module over $T$ annihilated by a product of powers of elements of $P.$ Then \[\lim_{n\to\infty}\mathbb{P}\left(\mathrm{cok}(M(n))\otimes \prod_{i=1}^kT_{\p_i}\simeq N\right)=\frac{1}{|\mathrm{Aut}(N)|}\prod_{i=1}^k\prod_{j=1}^\infty(1-|T/\mathfrak{p}_i|^{-j}).\]

\end{thm}

We will see in section 3 that these probabilities sum to 1. 
A simple example of a random matrix with entries in $\mathcal{O}_K$ for some number field $K$ that has a cokernel satisfying the above distribution is the following: let $\{a_1,\dots,a_d\}$ be an integral basis of $\mathcal{O}_K$, and let $0<p_0,p_1,\dots,p_d<1$ be such that $p_0+\cdots+p_d=1.$ Then let $M(n)$ be the random matrix with independent entries that take the value 0 with probability $p_0$ and $a_i$ with probability $p_i.$ Despite having entries concentrated on just a small finite set $\{0,a_1,\dots,a_d\},$ these $M(n)$ have a cokernel with the same distribution as matrices that are much more equidistributed in the quotients $\mathcal{O}_K/p\mathcal{O}_K.$ When $K=\Q,$ this example reduces to the statement that a random matrix $M(n)$ with independent entries distributed $M(n)_{ij}\sim \mathrm{Bern}(q)$ for some $q>0$ has a cokernel distributed according to (\ref{Zdist}).

The need for stronger assumptions is apparent over many other rings as well. If we consider random matrices with entries valued in some function field $K[x]$ with $\mathrm{char}(K)\neq 2,$ any constant matrix will be invariant under the nontrivial isomorphism $\tau: x\mapsto -x$, which again restricts the cokernel of such a matrix to those that are invariant under $\tau.$ One can even allow the matrix entries to be supported on many more values, such as all even polynomials in $x,$ and the same problem would arise. Therefore, the exact conditions needed to arrive at a Cohen-Lenstra type distribution on the cokernels of random matrices over more general Dedekind domains must be more complex than simply counting the number of residues achieved in each quotient. The precise statement of the main condition on matrix entries used in this paper is as follows.

\begin{defn}
(a) Let $V$ be an $\mathbb{F}_p$-vector space and $y$ be a random variable valued in V. Then $y$ is $\mathbf{\epsilon}$\textbf{-balanced} if $P(y \in H)\leq1-\epsilon$ for every proper affine $\mathbb{F}_p$-subspace $H\subset V$. A random vector or random matrix with entries in $V$ is $\mathbf{\epsilon}$\textbf{-balanced} if the entries are independent and $\epsilon$-balanced.

(b) Let $T$ be a countable Dedekind domain such that $T/J$ is finite for all nonzero ideals $J\subset T,$ equipped with the discrete $\sigma$-algebra. Let $\mathcal{I}$ denote the set of all nonzero ideals of $T$ such that $T/I$ is an elementary abelian group, i.e. that $T/I\simeq (\Z/p\Z)^k$ for some prime $p$ and integer $k>0$, as an abelian group. Let $\epsilon\in \R_{>0}^{\mathcal{I}}$ and write $\epsilon=(\epsilon_I)_{I\in\mathcal{I}}.$ Then a random variable $y$ valued in $T$ is $\mathbf{\epsilon}$\textbf{-balanced} if for every $I\in\mathcal{I},$ the image of $y$ in $T/I$ is an $\epsilon_I$-balanced random variable. A random vector or random matrix with entries in $T$ is $\epsilon$-balanced if the entries are independent and $\epsilon$-balanced.
\end{defn}

Motivation for the definition is provided in section 2. Using this definition, we have the following result for countable Dedekind domains for which the quotient by any nonzero proper ideal is finite.


\begin{thm}
Let $T$ be a countable Dedekind domain for which the quotient by any nonzero proper ideal is finite, equipped with the discrete $\sigma$-algebra. Let $P=\{\p_1,\dots,\p_k\}$ be a finite set of nonzero prime ideals of $T.$ Let $N$ be a finite module over $T$ annihilated by a product of powers of elements of $P.$ Let $u$ be a positive integer and $M(n)$ be an $\epsilon$-balanced $n\times (n+u)$ random matrix with entries valued in $T$, for some $\epsilon\in\R^{\mathcal{I}}_{>0}$, where $\mathcal{I}$ is the set of all nonzero ideals of $T$ such that $T/I$ is an elementary abelian group. Then:
\[\lim_{n\to\infty}\mathbb{P}\left(\mathrm{cok}(M(n))\otimes \prod_{i=1}^kT_{\p_i}\simeq N\right)=\frac{1}{|\mathrm{Aut}(N)||N|^u}\prod_{i=1}^k\prod_{j=1}^\infty (1-|T/\p_i|^{-u-j}).\]
\end{thm}

We require $T$ to be countable so we can use the discrete $\sigma$-algebra and avoid measurability concerns. We can also allow for matrices $M(n)$ with entries valued in $T_\p=\varprojlim T/\p^n$ for some prime ideal $\p\subset T,$ using the Borel $\sigma$-algebra generated by the $\p$-adic topology, in which case an analagous statement holds.

Considering modules over polynomial rings $\mathbb{F}_p[x],$ this result implies that we can sample the entries of $M(n)$ independently from some given distribution supported on $\{1,x,x^2,\dots\},$ and arrive at the asymptotic cokernel distribution given above. 

Given any maximal ideal $\p\subset T,$ we have that $\mathrm{cok}(M(n))\otimes T/\p$ is the cokernel of $M$ reduced mod $\p,$ so the entries can be viewed as valued in a finite field. As such, the results in this paper can be seen as an enhancement of the long history of work on singularity and ranks of the random integral matrices considered over finite fields, including the results of \cite{CRR90,KK01,KL75,Koz66,Map10}. In particular, the notion of ``not too concentrated" that our $\epsilon$-balanced condition on our random matrices ensures is closely resembles that of \cite{KK01}.

We now give an overview of the paper. To determine the distribution of the cokernels described above, we first compute their moments, and then apply results from $\cite{SW22}$ to compute the distribution. This technique has been used in many papers determining universality properties of random groups, including \cite{Lee22, Mes20,NP22, NW21,NW22, Woo17, Woo19}. We are interested in the moments $\mathbb{E}(\#\mathrm{Sur}(\mathrm{cok}(M(n)),N))$ for any $T$-module $N$, and we will find these by proving inverse Littlewood-Offord theorems in section 2. Despite our end result being a statement about certain random modules (viewed as cokernels random matrices with entries in a Dedekind domain $T$), in the proofs of the inverse Littlewood-Offord theorems, we will make extensive use of the underlying abelian group structure of $T$ and its quotients. The abelian group structure is in fact the primary motivation behind our precise formulation of the notion that the matrix entries cannot be ``too concentrated." Finally, in section 3, to determine the cokernel distribution, we apply universality results from \cite{SW22} that utilize the aforementioned moments.

\subsection{Notation and Terminology}
All vector spaces in this paper will be considered over $\mathbb{F}_p$ for some rational prime $p,$ i.e. even if a space is a vector space over some $\mathbb{F}_q,$ where $q=p^k$ for some positive integer $k,$ we will still consider it as a vector space over $\F_p.$ The term \emph{proper affine subpace} of a vector space $V$ over $\F_p$  refers to a subset $H\subset V$ such that $H=a+H'$ for a proper subspace $H'\subset V,$ as a vector subspace over $\mathbb{F}_p$, and some $a\in V.$ Additionally, the cokernel of a matrix $M\in M_{n\times n}(R)$, denoted $\mathrm{cok}(M)$, is the $R$-module $R^n/M(R^n).$ For any $R$-modules $A,B$, write $\mathrm{Hom}(A,B)$ and $\mathrm{Sur}(A,B)$ to denote the set of $R$-module homomorphisms and surjective homomorphisms from $A$ to $B$, respectively. Write $\mathrm{Hom}_\Z(A,B)$ and $\mathrm{Sur}_\Z(A,B)$ to denote the set of abelian group homomorphisms and surjective abelian group homomorphisms from $A$ to $B,$ respectively. Finally, write $[n]$ to denote $\{1,\dots,n\}.$

\section{Computing the Moments}

Let $T$ be a Dedekind domain such that $T/J$ is finite for all nonzero ideals $J\subset T$. If $T$ is countable, equip it with the discrete $\sigma$-algebra. If $T$ is not countable, assume it is the completion $T=\varprojlim S/\p^n$ for some maximal ideal $\p$ of a Dedekind domain $S$ such that $S/J$ is finite for all nonzero ideals $J\subset S$. In this case, equip $T$ with the Borel $\sigma$-algebra generated by the $\p$-adic topology.

Our high level strategy in this section, including notions of codes and depth, follows \cite{Woo19}, but many details require new ideas for this general case.

In the case of random integral matrices, it was required that the matrix entries be independent and ``not too concentrated." Definition 1.2 continues to impose the conditions that the matrix entries be independent, and extends the notion of ``not too concentrated" to countable Dedekind domains.

The definition departs from the notion of $\epsilon$-balanced in the case of random integral matrices in a couple of subtle but significant ways. The distribution of cokernels of random integral matrices was calculated for matrices whose entries were $\epsilon$\emph{-balanced} for some $\epsilon$, i.e. that $\mathbb{P}(y\equiv r\Mod \p)\leq 1-\epsilon$ for every maximal ideal $\p\subset \Z$ and every $r\in \Z/\p.$ In other words, each entry need only be supported by two values in $T/\p$ for every maximal ideal $\p$.

In more general cases, as shown earlier, merely assuming that the matrix entries are supported on two values in $T/\p$ for every maximal ideal $\p\subset T$ is not sufficient to ensure the resulting matrix has a cokernel following the distribution given in Theorem 1.3, hence the more robust assumption given in the definition. For an ideal $I$ such that $T/I$ is a $\F_p$-vector space for some rational prime $p$, our assumption implies that an $\epsilon$-balanced random variable $y$ valued in $T$ is supported on at least $\dim_{\F_p}(T/I)+1$ values when projected down to $T/I$. In other words, the support of $y$ consists of at least $\dim_{\F_p}(T/I)+1$ values, no pair of which are congruent modulo $I.$ However, since the dimension of $T/I$ can become arbitrarily large as $I$ varies, we cannot make the assumption that a single $\epsilon$ determine the degree to which our random variable can be concentrated in some $T/I$ regardless of which $I$ is chosen. Therefore our measure of how ``unconcentrated" our random variable is, $\epsilon$, must vary with $I$, which is why we have defined it as an element of $\R_{\geq0}^{\mathcal{I}}.$

We will study random matrices with entries valued in $T$ by reducing them mod $a$, for all nonzero, non-unit elements $a\in T.$ Throughout this section we will work with the following objects. Let $a\in T$ be nonzero and non-unit, and let $R=T/aT.$ We will study random matrices with entries valued in $R$. These reductions of $T$-valued random matrices are sufficient to establish our results because for a random matrix $M(n),$ as defined in the introduction, we will determine the asymptotic distribution of the reductions $\mathrm{cok}(M(n))\otimes \prod_{i=1}^kT_{\p_i}$ (for any fixed finite set of primes $\{\p_1,\dots,\p_k\}$), which can be studied by reducing the entries of $M(n)$ first. Since $R$ is finite, there is a finite set $\mathcal{I}$ of ideals $I$ such that $R/I$ is an elementary abelian group. Thus, when considering random variables valued in $R,$ it makes sense to adapt our notion of $\epsilon$-balanced slightly, as follows.

\begin{defn}
Let $T$ be a countable Dedekind domain such that $T/J$ is finite for all nonzero ideals $J\subset T.$ Let $R=T/aT$ for some nonzero, nonunit $a\in T.$ Let $\epsilon_R>0$ be a real number. A random variable $y$ valued in $R$ is $\mathbf{\epsilon_R}$\textbf{-balanced} if, for every ideal $I$ such that $R/I$ is an elementary abelian group, the image of $y$ in $R/I$ is an $\epsilon_R$-balanced random variable.
\end{defn}

In this section, we consider an $\epsilon_R$-balanced random matrix $M$ with entries in $R.$ Additionally, fix a non-negative integer $u$ and study random $n\times (n+u)$ matrices $M$ with entries valued in $R.$ Let $M_1,\dots,M_{n+u}$ be the columns of $M$ (which are random vectors valued in $R^n$) and $M_{ij}$ be the entries of $M.$ Let $V=R^n$ with standard basis $v_i$ and $W=R^{n+u}$ with standard basis $w_j.$ For every $\sigma\subset [n]$, denote by $V_\sigma$ the distinguished submodule generated by the $v_i$ with $i\notin \sigma.$ Additionally, view every matrix $M$ as an element of $\Hom(W,V),$ and its columns $M_j$ as elements of $V$ so that $M_j=Mw_j=\sum_iM_{ij}v_i.$ Let $N$ be a finite $R$-module (equivalently, a $T$-module such that $aN=0$). We have $\mathrm{cok}(M)=V/M(W).$

To investigate the moments $\mathbb{E}(\#\mathrm{Sur}(\mathrm{cok}(M),N))$, observe that each such surjection lifts to a surjection $V\to N,$ so we have that \[\mathbb{E}(\#\mathrm{Sur}(\mathrm{cok}(M),N))=\sum_{F\in\mathrm{Sur}(V,N)}\mathbb{P}(F(M(W))=0).\]

Since $M$ is $\epsilon_R$-balanced, the columns are independent and we have \[\mathbb{P}(F(M(W))=0)=\prod_{j=1}^{n+u}\mathbb{P}(F(M_j)=0).\]

We will estimate the probabilities $\mathbb{P}(F(M_j)=0).$ We will separate the surjections $F$ into cases and establish estimates in each case separately. The first set of surjections $F$ we consider satsifies the following property.

\begin{defn}
Given an integer $n\geq 1,$ let $V=R^n$. Let $N$ be a finite $R$-module. We say that $F\in \Hom_R(V,N)$ is a \textbf{code of distance} $w$ if for every $\sigma\subset[n]$ with $|\sigma|<w,$ we have $FV_\sigma=N.$
\end{defn}

We introduce the following notion that will be helpful for Lemmas 2.4 and 2.5.

\begin{defn}
A random variable $y$ valued in a finite ring $S$ is $\mathbf{\epsilon}$\textbf{-nonconstant} if $\mathbb{P}(y=s)\leq 1-\epsilon$ for all $s\in S.$ 
\end{defn}

\begin{lemma}
Let $N$ be a finite $R$-module. Let $\delta>0$ be a real number. Let $n$ be a positive integer. Let $X$ be an $\epsilon_R$-balanced random vector valued in $V=R^n.$ Let $F\in\Hom(V,N)$ be a code of distance $\delta n$ and let $A$ be an element of $N.$ Then  \[\left|\mathbb{P}(FX=A)-|N|^{-1}\right|\leq\exp(-\epsilon_R\delta n/|R|^2).\]
\end{lemma}

The proof of this lemma uses the discrete Fourier transform and the following helpful estimate, which is almost identical to Lemma 2.2 in \cite{Woo19}, so we will omit the proof.

\begin{lemma}[Lemma 2.2 from \cite{Woo19}]
Let $\epsilon>0$ be a real number and $m\geq2$ an integer. Let $\zeta$ be a primitive $m$th root of unity. Let $y$ be an $\epsilon$-nonconstant random variable valued in $\Z/m\Z.$ Then $|\E(\zeta^{y})|\leq\exp(-\epsilon/m^2).$
\end{lemma}

\begin{proof}[Proof of Lemma 2.4]
Let $m=|R|$ and $\zeta$ be a primitive $m$th root of unity. By the discrete Fourier transform, we have that
\begin{align*}
\mathbb{P}(FX=A)
&=|N|^{-1}\sum_{C\in\Hom_\Z(N,\Z/m\Z)}\E\left(\zeta^{C(FX-A)}\right)\\
&= |N|^{-1}+|N|^{-1}\sum_{C\in\Hom_\Z(N,\Z/m\Z)\backslash \{0\}}\E(\zeta^{C(-A)})\prod_{1\leq j\leq n}\E(\zeta^{C(F(v_j)X_j)}).
\end{align*}

Fix some $C\in \Hom_\Z(N,\Z/m\Z)\setminus \{0\}.$ Our goal is to show there are at least $\delta n$ values of $j$ such that $C(F(v_j)X_j)$ is an $\epsilon_R$-nonconstant variable in $\Z/m\Z$, after which the bound in Lemma 2.5 implies the deisred result immediately. To do so, for each $j$ we will view the mapping $x\mapsto C(F(v_j)x)$ as a mapping from $R\to \Z/m\Z$ and examine its kernel.

For each $j$, note that we have an $R$-homomorphism $\psi_j:R\to N$ given by $\psi_j(1)=F(v_j),$ so we can define a $\Z$-homomorphism  $\varphi_j : R\to \Z/m\Z$ given by $\varphi_j=C\circ\psi_j.$ Thus $\varphi_j(X_j)=C(F(v_j)X_j).$ We claim that there must be at least $\delta n$ values of $j$ such that $\varphi_j$ is a nonzero map. Suppose this is not the case. Then there is some $\sigma\subset[n]$ with $|\sigma|<\delta n$ such that $\varphi_j=0$ for all $j\notin \sigma.$ Therefore $\mathrm{Im}(\psi_j)\subset \ker(C)$ for all $j\notin \sigma,$ so $F(V_\sigma)\subset \ker(C)$. But $C$ is nonzero, so $\ker(C)$ is a proper subgroup of $N$, so $F(V_\sigma)$ is a proper subgroup of $N$, contradicting our assumption that $F$ is a code.

It now suffices to show that for $j$ such that $\varphi_j$ is nonzero, $\varphi_j(X_j)=C(F(v_j)X_j)$ is an $\epsilon_R$-nonconstant random variable valued in $\Z/m\Z$. Let $S=\ker(\varphi_j).$ Because $\varphi_j$ is nonzero, $S$ is a proper subgroup of $R$, so $S$ must be contained in some maximal proper subgroup $M$ of $R.$ Because $M$ is maximal, we have by the correspondence theorem that $R/M\simeq \Z/\ell \Z$, as an abelian group, for some rational prime $\ell.$ Thus we have a quotient map $\pi:R\mapsto \Z/\ell\Z.$ Because $R/S\simeq \mathrm{Im}(\varphi_j)$ and $S\subset M,$ it suffices to show that $\pi(X_j)$ is an $\epsilon_R$-nonconstant variable in $\Z/\ell\Z$. Observe that $\pi$ factors through $R/\ell R$ because $\ell R \subset\ker(\pi)$, so we have a map $\overline{\pi}:R/\ell R\to\Z/\ell\Z.$ Thus for every proper affine subspace $H$ of $R/\ell R$, we have that $\mathbb{P}(\overline{X_j}\in H)\leq 1-\epsilon_R$, where  $\overline{X_j}$ is the image of $X_j$ under the projection $R\mapsto R/\ell R$. Since $\pi(X_j)=\overline{\pi}(\overline{X_j})$, and $\ker(\overline{\pi})$ is proper subspace of $R/\ell R$, this implies that $\pi(X_j)$ is an $\epsilon_R$-nonconstant in $\Z/\ell\Z.$


Thus we have shown that at least $\delta n$ values of $j$ such that $C(F(v_j)X_j)$ is an $\epsilon_R$-nonconstant variable in $\Z/m\Z$. We can then apply Lemma 2.5 to conclude: 
\[\left|\E(\zeta^{C(-A)})\prod_{1\leq j\leq n}\E(\zeta^{C(F(v_j)X_j)})\right|\leq \exp(-\epsilon_R\delta n/m^2)\] for $C\in\Hom_\Z(G,\Z/m\Z)\backslash \{0\}.$ The result follows.
\end{proof}

We can put the estimates for columns together as follows:

\begin{lemma}
Let $R,V,N$ be as above. Let $u$ be a non-negative integer. Let $\epsilon_R,\delta>0$ be real numbers. Then there are real numbers $c,K>0$ depending on $J,N,u,\epsilon_R,$ and $\delta,$ such that for every positive integer $n,$ every $\epsilon_R$-balanced random matrix $M$ valued in $\Hom_R(W,V),$ every code $F\in\Hom_R(V,N)$ of distance $\delta n$, and every $A\in \Hom_R(W,N)$, we have \[|\mathbb{P}(FM=A)-|N|^{-n-u}|\leq\frac{K\exp(-cn)}{|N|^{n+u}}\] where $W=R^{n+u}.$
\end{lemma}

Lemma 2.6 can be shown by the same argument as Lemma 2.4 in \cite{Woo19}, as going from estimates for columns to estimates for matrices required only inequalities about real numbers.

We now need to deal with $F$ that are not codes of distance $\delta n$. Such $F\in \Hom_R(V,N)$ can be categorized based on the largest submodule of $N$ they are a code for of distance at least $\delta n$. Note that $V_\sigma$ is a submodule of $V$ for any $\sigma$, so $F(V_\sigma)$ is a submodule of $N.$ For a positive integer $D$ with prime factorization $D=\prod_ip_i^{e_i}$ define $\ell(D)=\sum_i e_i.$ 

\begin{defn}
For a real $\delta>0$, the $\mathbf{\delta}$\textbf{-depth} of an $F\in\Hom_R(V,N)$ is the maximal positive integer $D$ such that there is a subset $\sigma\subset[n]$ with $|\sigma|<\ell(D)\delta n$ such that $D=[N:F(V_\sigma)]$, or is $1$ if there is no such $D.$
\end{defn}

Observe that when the $\delta$-depth of $F$ is 1, we have that for all $\sigma\subset [n]$ satisfying $|\sigma|<\delta n,$ $F(V_\sigma)=G,$ implying that $F$ is a code of distance $\delta n.$ It turns out that there are far fewer non-codes than codes. In fact, we can bound the number of $F\in\Hom_R(V,G)$ that are of $\delta$-depth $D$ using the following.

\begin{lemma}
Let $V,N$ be as above. There is a constant $K,$ depending on $a$ and $G,$ such that for all positive integers $n$, all integers $D>1,$ and all real numbers $\delta>0,$ the number of $F\in\Hom_R(V,G)$ of $\delta$-depth $D$ is at most \[K\binom{n}{\lceil \ell(D)\delta n\rceil-1}|G|^n|D|^{-n+\ell(D)\delta n}.\] 
\end{lemma}
\begin{proof}
This is the same as Lemma 5.2 in \cite{Woo17}, replacing subgroups with submodules. Since submodules are also subgroups, this bound still holds.
\end{proof}

For each $\delta$-depth, we obtain the following estimate.

\begin{lemma}
Let $V,N$ be as above. If $F\in\Hom_R(V,N)$ has $\delta$-depth $D>1$ and $[N:F(V)]<D,$ then for all $\epsilon_R$-balanced random vectors $X$ valued in $V,$ \[\mathbb{P}(FX=0)\leq (1-\epsilon_R)(D|N|^{-1}+\exp(-\epsilon_R\delta n/|R|^2)).\]
\end{lemma}
\begin{proof}
The proof is the direct analog of the proof of Lemma 2.7 in \cite{Woo19} for random modules, but we include it for completeness. Let $V$ have standard basis $v_i,$ and pick $\sigma\subset[n]$ with $|\sigma|<\ell(D)\delta n$ such that $D=[N:F(V_\sigma)].$ Let $F(V_\sigma)=H.$ Since $[N:F(V)]<D$, the set $\sigma$ is non-empty. We have that:
\begin{align*}
    \mathbb{P}(FX=0)&=\mathbb{P}\left(\sum_{j\in\sigma}F(v_j)X_j\in H\right)\\
    &\times \mathbb{P}\left(\sum_{j\notin\sigma}F(v_j)X_j=-\sum_{j\in\sigma}F(v_j)X_j\;\middle|\; \sum_{j\in\sigma}F(v_j)X_j\in H\right).
\end{align*}


We claim that the first factor $\mathbb{P}\left(\sum_{j\in\sigma}F(v_j)X_j\in H\right)\leq1-\epsilon.$ Since $[N:F(V)]<D,$ there must be some $i\in\sigma$ with the reduction $F(v_i)\neq 0\in N/H.$ Suppose we condition on all other $X_k$ for $k\neq i.$ Since $N/H$ is a also finite $T$-module and $T$ is a Dedekind domain, we can write $N/H=T/\p_1^{e_1}\times\cdots\times T/\p_k^{e_k}$ for some prime ideals $\p_j\in T$ by the classification of modules over a Dedekind domain. Treating each $F(v_j)$ as an element of $T/\p_1^{e_1}\times\cdots\times T/\p_k^{e_k},$ we can write  $F(v_i)=(a_1,\dots,a_k)$, where $a_j\in T/\p_j^{e_j}$ for all $j$. Additionally, there is some $\ell$ such that $a_\ell\neq0$ since $F(v_i)\neq0.$ Then given all $X_k$ for $k\neq i,$ in order for $\sum_{j\in\sigma}F(v_j)X_j$ to be 0 in $N/H,$ we have that $F(v_i)X_i$ is a uniquely determined nonzero element of $N/H$, so $a_\ell X_i$ is a uniquely determined element of $T/\p_\ell^{e_\ell}.$ It follows that $X_i$ can take on at most one value mod $\p_\ell$. The probability of this happening is at most $1-\epsilon_R$ since $X_i$ is $\epsilon_R$-balanced and every point is an affine subspace. Thus $$\mathbb{P}\left(\sum_{j\in\sigma}F(v_j)X_j\in H\right)\leq1-\epsilon_R$$ 

For the second factor, note that the restriction of $F$ to $V_\sigma$ is a code of distance $\delta n$ in $\Hom_R(V_\sigma,H).$ It follows from Lemma 2.4 that \[\mathbb{P}\left(\sum_{j\notin\sigma}F(v_j)X_j=-\sum_{j\in\sigma}F(v_j)X_j\;\middle|\; \sum_{j\in\sigma}F(v_j)X_j\in H\right)\leq |H|^{-1}+\exp(-\epsilon_R\delta n/m^2).\]

The lemma follows.
\end{proof}

We can extend this estimate for vectors to an estimate for matrices as follows.
\begin{lemma}
Let $V,N$ be as above. Let $u$ be a non-negative integer. Let $\epsilon_R,\delta>0$ be real numbers. Then there is a real number $K$, depending on $a,N,u,\epsilon_R,\delta$, such that for every positive integer $n,$ every $\epsilon_R$-balanced random matrix $M$ valued in $\Hom(W,V),$ every $F\in \Hom_R(V,N)$ of $\delta$-depth $D>1$ with $[N:F(V)]<D$, we have \[\mathbb{P}(FM=0)\leq K\exp(-\epsilon_R n)D^n|N|^{-n}\] here $W=R^{n+u}.$
\end{lemma}
The proof of Lemma 2.10 is the same as the proof of Lemma 2.8 in \cite{Woo19}.

We can then conclude the following theorem.
\begin{thm}
Let $N$ be a finite $R$-module. Let $u$ be a non-negative integer. Let $\epsilon_R>0$ be a real number. Then there are $c,K>0$ depending on $a,N,u,\epsilon_R,$ such that for any positive integer $n,$ and any $\epsilon_R$-balanced random matrix $M$ with entries in $R,$ we have \[|\mathbb{E}(\#\mathrm{Sur}(\mathrm{cok}(M),N))-|N|^{-u}|\leq Ke^{-cn}.\]
\end{thm}
\begin{proof}
Let $W=R^{n+u}$, so we can view $M\in\Hom_R(W,V),$ so $\mathrm{cok}(M)=V/M(W).$ Since each surjection $\mathrm{cok}(M)\to N$ lifts to a surjection $V\to N,$ we have that \[\mathbb{E}(\#\mathrm{Sur}(\mathrm{cok}(M),N))=\sum_{F\in\mathrm{Sur}(V,N)}\mathbb{P}(F(M(W))=0).\]

Our goal is to bound the following:
\begin{align*}
    \left|\left(\sum_{F\in\mathrm{Sur}(V,N)}\mathbb{P}(FM=0)\right)-|N|^{-u}\right|&= \left|\left(\sum_{F\in\mathrm{Sur}(V,N)}\mathbb{P}(FM=0)\right)-\left(\sum_{F\in\Hom(V,N)}|N|^{-n-u}\right)\right|\\
    &\leq \sum_{\substack{F\in\mathrm{Sur}(V,N)\\ F \text{ code of distance }\delta n}}|\mathbb{P}(FM=0)-|N|^{-n-u}|\\&+\sum_{\substack{F\in\mathrm{Sur}(V,N)\\ F \text{ not code of distance }\delta n}}\mathbb{P}(FM=0)+\sum_{\substack{F\in\Hom(V,N)\\ F \text{ not code of distance }\delta n}}|N|^{-n-u}.
\end{align*}

We claim that each term can be bounded by $Ke^{-cn}$ for some constants $c,K$ depending only on $a,N,u,\epsilon_R,\delta,$ and $d.$ In the remainder of the proof, $K$ may change as needed so long as it only depends on $a,N,u,\epsilon_R,\delta,$ and $d.$

Before bounding each term, we pick some constants. Pick a real $d<\min(\epsilon_R,\log(2)).$ Given $N$ and $d,$ pick a real number $\delta>0$ sufficiently small such that there exists $K$, such that for all $n\in\Z^+$ we have \[\binom{n}{\lceil \ell(D)\delta n\rceil-1}|N|^{\ell(|N|)\delta n}\exp(-\epsilon_R n)\leq Ke^{-dn}\] and \[\binom{n}{\lceil \ell(D)\delta n\rceil-1}2^{-n+\ell(|N|)\delta n}\leq Ke^{-dn}.\]

The bounding of the first term \[\sum_{\substack{F\in\Hom(V,N)\\ F \text{ code of distance }\delta n}}|\mathbb{P}(FM=0)-|N|^{-n-u}|\leq Ke^{-cn}\] follows from Lemma 2.6. We can also establish the following:
\begin{align*}
\sum_{\substack{F\in\mathrm{Sur}(V,N)\\ F \text{ not code of distance }\delta n}}\mathbb{P}(FM=0)&=\sum_{\substack{D>1\\D|\#N}}\sum_{\substack{F\in\mathrm{Sur}(V,N)\\ F \text{ }\delta-\text{depth }D}}\mathbb{P}(FM=0)\\
&\leq \sum_{\substack{D>1\\D|\#N}}K\binom{n}{\lceil \ell(D)\delta n\rceil-1}|N|^nD^{-n+\ell(D)\delta n}\exp(-\epsilon_R n)D^n|N|^{-n}\\
&\leq\sum_{\substack{D>1\\D|\#N}}K\binom{n}{\lceil \ell(D)\delta n\rceil-1}D^{\ell(D)\delta n}\exp(-\epsilon_Rn)\\
&\leq K\binom{n}{\lceil \ell(D)\delta n\rceil-1}|N|^{\ell(|N|)\delta n}\exp(-\epsilon_R n)\\
&\leq Ke^{-dn}.
\end{align*}

This gives us the bound on the middle term.

For the last term, notice that
\begin{align*}
\sum_{\substack{F\in\mathrm{Sur}(V,N)\\ F \text{ not code of distance }\delta n}}|N|^{-n-u}&=\sum_{\substack{D>1\\D|\#N}}\sum_{\substack{F\in\mathrm{Sur}(V,N)\\ F \text{ }\delta-\text{depth }D}}|N|^{-n-u}\\
&\leq \sum_{\substack{D>1\\D|\#N}}K\binom{n}{\lceil \ell(D)\delta n\rceil-1}|N|^n|D|^{-n+\ell(D)\delta n}|N|^{-n}\\
&\leq K\binom{n}{\lceil \ell(D)\delta n\rceil-1}2^{-n+\ell(|N|)\delta n}\\
&\leq Ke^{-dn}.
\end{align*}

Additionally,
\begin{align*}
    \sum_{F\in\Hom(V,N)\setminus\mathrm{Sur}(V,N)}|N|^{-n-u}&\leq\sum_{H \text{ proper submodule of }N}\sum_{F\in\Hom(V,H)}|N|^{-n-u}\\
    &\leq \sum_{H \text{ proper submodule of }N}|H|^{n+u}|N|^{-n-u}\\
    &\leq Ke^{-dn}.
\end{align*}

The bound on the last term follows.
\end{proof}
\section{Distribution of Cokernels}
We will combine the results of Theorem 1.2 and Lemma 6.3 from $\cite{SW22}$ to determine the asymptotic distribution of $\mathrm{cok}(M(n))$ as $n\to \infty$ from the moments given in Theorem 2.12, as follows.
\begin{thm}[Lemma 6.3 from \cite{SW22}]
Let $R$ be a ring. Let $u$ be a real number and for each finite $R$-module $N$ let $M_N=|N|^{-u}$. Let $S$ be a finite quotient ring of $R$, and let $K_1,\dots, K_n$ be representatives of the isomorphism classes of finite simple S-modules. Let $q_1,\dots, q_n$ be the cardinalities of the endomorphism fields of $K_i$. Let $\{X_n\}$ be a sequence of finite $S$-modules. Suppose for every finite $S$-module $N$ we have $\lim_{n\to\infty}\mathbb{E}(\#\mathrm{Sur}(X_n,N))=M_N.$ Then \[\lim_{n\to \infty}\mathbb{P}(X_n\simeq N)=\frac{1}{|\mathrm{Aut}(N)||N|^u}\prod_{i=1}^n\prod_{j=1}^\infty\left(1-\frac{|\mathrm{Ext}_S^1(N,K_i)|}{|\mathrm{Hom}(N,K_i)||K_i|^u}|q_i|^{-j}\right).\]
\end{thm}

Theorem 2.11, along with Theorem 3.1, implies the following, which was Theorem 1.3 in the introduction.

\begin{thm}
Let $T$ be a countable Dedekind domain for which the quotient by any nonzero proper ideal is finite. Let $P=\{\p_1,\dots,\p_k\}$ be a finite set of nonzero prime ideals of $T.$ Let $N$ be a finite $T$-module annihilated by a product of powers of elements of $P.$ Let $u$ be a positive integer and $M(n)$ be an $\epsilon$-balanced $n\times (n+u)$ random matrix with entries valued in $T$, for some $\epsilon\in\R^{\mathcal{I}}_{>0}$, where $\mathcal{I}$ is the set of all nonzero ideals of $T$ such that $T/I$ is an elementary abelian group. Then:
\[\lim_{n\to\infty}\mathbb{P}\left(\mathrm{cok}(M(n))\otimes \prod_{i=1}^kT_{\p_i}\simeq N\right)=\frac{1}{|\mathrm{Aut}(N)||N|^u}\prod_{i=1}^k\prod_{j=1}^\infty (1-|T/\p_i|^{-u-j}),\] where for all $i,$  $T_{\p_i}=\varprojlim T/\p_i^n.$
\end{thm}
\begin{proof}
Applying Theorem 3.1, let $R=T.$ Let $n_1,\dots,n_k$ be the minimal positive integers such that $N$ is a $T/\p_1^{n_1}\times\cdots\times T/\p_k^{n_k}$-module. Let $S=T/\p_1^{n_1+1}\times\cdots\times T/\p_k^{n_k+1}.$ Viewing $M(n)$ as a matrix with entries valued in $S,$ it has cokernel $\mathrm{cok}(M(n))\otimes S.$ Then by Theorem 2.12, we have that $\{\mathrm{cok}(M(n))\otimes S\}$ is a sequence of random $S$-modules such that $\lim_{n\to\infty}\mathbb{E}(\#\mathrm{Sur}(\mathrm{cok}(M(n))\otimes S,N))=|N|^{-u}.$ Furthermore, we know that $T/\p_1,\dots, T/\p_k$ are the representatives of the isomorphism classes of finite simple $S$-modules, and $q_i=|T/\p_i|$ for all $i.$

We claim that $|\mathrm{Ext}_S^1(N,T/\p_i)|=|\mathrm{Hom}(N,T/\p_i)|.$ Indeed, consider the free resolution of $N=T/\p_1^{a_{11}}\times \cdots\times T/\p_1^{a_{1m_1}}\times\cdots\times T/\p_k^{a_{k1}}\times\cdots\times T/\p_k^{a_{km_k}}$ as an $S$-module, letting $m=m_1+\cdots+m_k$. 
\[\begin{tikzcd}
	\cdots & {S^{m}} & {S^m} & {S^m} & N
	\arrow[from=1-2, to=1-3]
	\arrow[from=1-3, to=1-4]
	\arrow[from=1-4, to=1-5]
	\arrow[from=1-1, to=1-2]
\end{tikzcd}\] Since all $a_{ij}\leq n_i<n_i+1,$ we have that the first two maps in the resulting cochain 
\[\begin{tikzcd}
	{\Hom_S(S^m,T/\mathfrak{p}_i)} & {\Hom_S(S^m,T/\mathfrak{p}_i)} & {\Hom_S(S^m,T/\mathfrak{p}_i)} & \cdots
	\arrow[from=1-1, to=1-2]
	\arrow[from=1-2, to=1-3]
	\arrow[from=1-3, to=1-4]
\end{tikzcd}\] are 0, so $|\mathrm{Ext}_S^1(N,T/\p_i)|=|\mathrm{Hom}(N,T/\p_i)|.$ Combining the above, we get that Theorem 3.1 implies \[\lim_{n\to\infty}\mathbb{P}\left(\mathrm{cok}(M(n))\otimes S\simeq N\right)=\frac{1}{|\mathrm{Aut}(N)||N|^u}\prod_{i=1}^k\prod_{j=1}^\infty (1-|T/\p_i|^{-u-j}),\] implying the desired result.
\end{proof}

An analogous statement holds for matrices $M(n)$ with entries valued in $T_\p$ for some prime ideal $\p\subset T.$

We remark that when $T$ is the ring of integers of some number field, the condition on $M(n)$ in Theorem 3.2 reduces to assuming that there is an $\epsilon>0$ such that for every rational prime $p$ and every proper affine subspace $H\subset T/pT$, we have $\mathbb{P}(M(n)_{ij}\in H)\leq 1-\epsilon$, implying Theorem 1.1.

Additionally, note that Theorem 3.2 establishes a sufficient condition to yield the distribution described above, but it is not a necessary one, as shown in (1.3) of \cite{CY23}.

When $T=\F_p[x]$, we require the existence of some $\epsilon\in \R_{>0}^{\F_p[x]^\times}$, so that for every nonzero $f\in \F_p[x]$ and every proper affine subspace $H\subset \F_p[x]/f$, $\mathbb{P}(M(n)_{ij}\in H)\leq 1-\epsilon(f).$ One example $M(n)$ that satisfies that hypothesis of Theorem 3.2 samples entries independently from some given distribution supported on $\{1,x,x^2,\dots\}.$

Finally, Lemma 6.8 in \cite{SW22} implies the following.
\begin{thm}
Let $T$ be a countable Dedekind domain for which the quotient by any nonzero proper ideal is finite. Let $P=\{\p_1,\dots,\p_k\}$ be a finite set of prime ideals of $T.$ Let $N$ be a finite $T$-module annihilated by a product of powers of elements of $P.$ Let $u$ be a positive integer and $M(n)$ be an $\epsilon$-balanced $n\times (n+u)$ random matrix with entries valued in $T$, for some $\epsilon\in\R^{\mathcal{I}}_{>0}$, where $\mathcal{I}$ is the set of all nonzero ideals of $T$ such that $T/I$ is an elementary abelian group. Then: \[\lim_{n\to\infty}\mathbb{P}\left(\mathrm{cok}(M(n))\otimes \prod_{i=1}^kT_{\p_i}\,\,\mathrm{ is}\,\mathrm{finite}\right)=1.\] In other words, the probabilities given in Theorem 3.2 sum to 1.
\end{thm}
\begin{proof}
Lemma 6.8 in \cite{SW22} directly implies the theorem when $P=\{\p_1\}$ contains only one prime. Factoring the probabilities in Theorem 3.2 over primes $\p_i\in P$ gives the desired result.
\end{proof}

\bibliographystyle{plain}
\bibliography{sources}
\end{document}